\documentclass[12pt,reqno]{amsart}
\usepackage{hyperref}
\usepackage{amsmath,amssymb,amsthm,fullpage,graphicx,cite}

\newtheorem{theorem}{Theorem}
\newtheorem{lemma}[theorem]{Lemma}

\newtheorem{proposition}[theorem]{Proposition}
\newtheorem{conjecture}[theorem]{Conjecture}
\theoremstyle{definition}

\theoremstyle{remark}
\newtheorem{remark}[theorem]{Remark}
	
\newcommand\N{\mathbb{N}}
\newcommand\R{\mathbb{R}}

\newcommand\C{\mathbb{C}}
\newcommand\E{\mathbb{E}}
\newcommand\cD{\mathcal{D}}

\newcommand\eps{\varepsilon}
\newcommand\ndv{\mathrel{\mkern1mu{\scriptstyle\not}\mkern-2.5mu\mid}}

\begin{document}

\title{Bounds on Rudin--Shapiro polynomials of arbitrary degree}

\author{Paul Balister}

\address{Department of Mathematical Sciences,
University of Memphis, Memphis, TN 38152, USA}

\email{pbalistr@memphis.edu}

\thanks{The author was partially supported by NSF grants DMS 1600742 and DMS 1855745.}

\begin{abstract}
 Let $P_{<n}(z)$ be the Rudin--Shapiro polynomial of degree $n-1$.
 We show that $|P_{<n}(z)|\le \sqrt{6n-2}-1$ for all $n\ge0$ and $|z|=1$,
 confirming a longstanding conjecture.
 This bound is sharp in the case when $n=(2\cdot 4^k+1)/3$ and $z=1$.
 We also show that for $n\ge m\ge0$, $|P_{<n}(z)-P_{<m}(z)|\le \sqrt{10(n-m)}$,
 which is asymptotically sharp in the sense that for any $\eps>0$
 there exists $n>m\ge0$ and $z$ with $|z|=1$ and
 $|P_{<n}(z)-P_{<m}(z)|\ge\sqrt{(10-\eps)(n-m)}$,
 contradicting a conjecture of Montgomery.
\end{abstract}

\maketitle

\section{Introduction}\label{s:introduction}

The Rudin--Shapiro polynomials $P_t$ and $Q_t$ are defined by setting $P_0(z)=Q_0(z)=1$
and, for $t\ge0$, inductively defining
\begin{align*}
 P_{t+1}(z)&=P_t(z) + z^{2^t}Q_t(z),\\
 Q_{t+1}(z)&=P_t(z) - z^{2^t}Q_t(z).
\end{align*}
These polynomials were introduced independently in the 1950s by
Shapiro~\cite[p.\thinspace39]{S52} and Rudin~\cite{R59} (although the sequence
$a_n$ of their coefficients was also previously studied by Golay~\cite{G49}), and have been extensively studied
over the last few decades, see e.g.~\cite{A95,B73,BC,BLM,MR,R17,S86,DH04}.

From the definition of $P_t$ we see that the first $2^t$ terms of $P_{t+1}$ are the same as for $P_t$, and
hence $P_t$ can be thought of as the first $2^t$ terms of an infinite power series
\[
 P_\infty(z):=\sum_{n=0}^\infty a_nz^n,
\]
where the coefficients $a_n\in\{-1,1\}$ can also be defined~\cite{B73} by the relations
\begin{equation}\label{e:sequence}
 a_0=1,\qquad a_{2n}=a_n,\qquad\text{and}\qquad a_{2n+1}=(-1)^na_n.
\end{equation}
Alternatively, writing $n=\sum_i b_i2^i$, $b_i\in\{0,1\}$, we have that~\cite{BC}
\[
 a_n=(-1)^{\sum_i b_ib_{i+1}},
\]
i.e., $a_n$ is determined by the parity of the number of `11's in the binary expansion of~$n$.

For $n\ge0$ write
\[
 P_{<n}(z):=\sum_{i=0}^{n-1}a_iz^i
\]
for the first $n$ terms of $P_\infty(z)$ so that, for $n>0$, $P_{<n}(z)$ is a polynomial of degree $n-1$, and
$P_t(z)=P_{<2^t}(z)$. For $n\ge m\ge0$ write
\[
 P_{[m,n)}(z):=P_{<n}(z)-P_{<m}(z)=\sum_{i=m}^{n-1}a_iz^i
\]
for the polynomial with $n-m$ terms consisting of the terms of $P_\infty(z)$
from $z^m$ to $z^{n-1}$.

Shapiro \cite{S52} has shown that for $|z|=1$, $|P_{<n}(z)|\le C\sqrt{n}$ for all $n$, where $C=2+\sqrt{2}\approx 3.41$, and
Saffari~\cite{S86} has sketched a proof that $C=(2+\sqrt{2})\sqrt{3/5}\approx 2.64$ suffices.
However, according to~\cite{M17} it has `long  been conjectured' that $C=\sqrt{6}\approx 2.45$
is sufficient, and indeed it is known that this is the best possible constant as
$|P_{<n}(1)|=2^{k+1}-1=\sqrt{6n-2}-1$ when $n=(2\cdot 4^k+1)/3$. In~\cite{A95} it is claimed that
Saffari proved this conjecture, but it appears that the proof is unpublished.
In this paper we give a proof of this conjecture in the following strong form.

\begin{theorem}\label{t:onedim}
 $|P_{<n}(z)|\le \sqrt{6n-2}-1$ for all\/ $n\ge1$ and\/ $|z|=1$.
\end{theorem}

In \cite{M17} Montgomery made the following conjecture about the polynomials $P_{[m,n)}$.

\begin{conjecture}\label{c:montgomery}
 $|P_{[m,n)}(z)|\le 3\sqrt{n-m}$ for all\/ $n\ge m\ge0$ and\/ $|z|=1$.
\end{conjecture}

The basis for this conjecture was numerical evidence that suggested the worst case was
when
\[
 m_k:=\frac{5\cdot 4^k+1}{3},\qquad n_k:=\frac{8\cdot 4^k+1}{3}
\]
and $z=1$, in which case $|P_{[m_k,n_k)}(1)|=3\cdot 2^k-2=3\sqrt{n_k-m_k}-2$.

Unfortunately this conjecture turns out to be false. The example polynomial is correct, but for large $k$
the largest value of $|P_{[m_k,n_k)}(z)|$ no longer occurs at $z=1$.
Indeed, it is not hard to show that
\[
 \lim_{k\to\infty}\frac{\big|P_{[m_k,n_k)}(e^{3\pi i/4})\big|^2}{n_k-m_k}=5+\tfrac{7}{\sqrt{2}}\approx  9.95,
\]
and even this is not the worst case when $k$ is very large.
Unfortunately the value of $z$ that maximizes $|P_{[m_k,n_k)}(z)|$ appears to be a highly
erratic function of~$k$, and so we are unable to give an explicit sequence $z_k$ with
$|P_{[m_k,n_k)}(z_k)|^2/(n_k-m_k)\to 10$. Nevertheless we show (in Section~\ref{s:dense}) that
\begin{equation}\label{e:max10}
 \lim_{k\to\infty}\frac{\sup_{|z|=1}|P_{[m_k,n_k)}(z)|^2}{n_k-m_k}=10.
\end{equation}
We also prove that this is asymptotically the worst case.

\begin{theorem}\label{t:twodim}
 $|P_{[m,n)}(z)|\le \sqrt{10(n-m)}$ for all\/ $n\ge m\ge0$ and all\/ $z$ with\/ $|z|=1$.
\end{theorem}

We prove Theorem~\ref{t:onedim} in Section~\ref{s:onedim} and Theorem~\ref{t:twodim} in Section~\ref{s:twodim}.
Equation~\eqref{e:max10} follows from Theorem~\ref{t:dense} below, which is a consequence of the proofs
of the results of Rodgers~\cite{R17} on the distribution of $P_t(z)/2^{(t+1)/2}$.
We prove Theorem~\ref{t:dense} and equation~\eqref{e:max10} in Section~\ref{s:dense}.

\section{The $L$-norm.}\label{s:norm}

We list a few  well-known properties of the polynomials $P_t$ and $Q_t$
which easily follow by induction, and can be found in, for example,~\cite{M17}.

\begin{proposition}\label{p:1} We have the following identities.
\begin{itemize}
\item[(a)] $|P_t(z)|^2+|Q_t(z)|^2=2^{t+1}$ for all $|z|=1$. In particular $|P_t(z)|,|Q_t(z)|\le 2^{(t+1)/2}$.
\item[(b)] $P_{t+k+1}(z)=P_k(z)P_t(z^{2^{k+1}})+z^{2^k}Q_k(z)P_t(-z^{2^{k+1}})$.\\
In particular $P_{t+1}(z)=P_t(z^2)+zP_t(-z^2)$.
\item[(c)] $Q_t(z)=(-1)^t z^{2^t-1}P_t(-z^{-1})$ and\/ $P_t(z)=(-1)^{t+1} z^{2^t-1}Q_t(-z^{-1})$.
\end{itemize}
\end{proposition}
Part (b) is particularly noteworthy as it shows that $P_\infty(z)$ is made up of alternate $\pm P_k$ and $\pm Q_k$ blocks,
namely $a_n z^{n2^k}P_k(z)$ for $n$ even and $a_n z^{n2^k}Q_k(z)$ for $n$ odd.

For $P\in\C[z,z^{-1}]$, define
\[
 \|P\|_\infty=\sup_{|z|=1}|P(z)|
\]
and\footnote{The subscript $L$ stands for `Limit', see Theorem~\ref{t:dense}.}
\[
 \|P\|_L=\sup_{|z|=1}\sqrt{|P(z)|^2+|P(-z)|^2}.
\]

\begin{lemma}
 $\|.\|_L$ is a norm on the vector space $\C[z,z^{-1}]$.
\end{lemma}
\begin{proof}
The fact that $\|P\|_L\ge0$ with equality iff $P=0$ is clear, so
it remains to prove that $\|P+Q\|_L\le\|P\|_L+\|Q\|_L$ for any $P,Q\in\C[z,z^{-1}]$. Now
\begin{align*}
 \big(|P(z)+Q(z)|^2+|P(-z)+Q(-z)|^2\big)^{1/2}
 &=  \big\|(P(z)+Q(z),P(-z)+Q(-z))\big\|_2\\
 &\le \big\|(P(z),P(-z))\big\|_2+\big\|(Q(z),Q(-z))\big\|_2\\
 &\le \|P\|_L+\|Q\|_L,
\end{align*}
where $\|(u,v)\|_2=\sqrt{|u|^2+|v|^2}$ is the standard $\ell_2$ norm on~$\C^2$
(here $z$ is fixed). The result now follows by taking the supremum over $|z|=1$.
\end{proof}

The advantage of $\|\cdot\|_L$ is that, unlike $\|\cdot\|_\infty$,
it scales well on Rudin--Shapiro polynomials, and thus allows us to effectively bound $P_{[m,n)}$ for
arbitrarily large $m$ and~$n$.

\begin{lemma}\label{l:scale}
 For any\/ $n\ge m\ge0$, $\|P_{[2m,2n)}\|_L^2=2\|P_{[m,n)}\|_L^2$ and\/ $\|P_{<2n}\|_L^2=2\|P_{<n}\|_L^2$.
\end{lemma}
\begin{proof}
By \eqref{e:sequence} (or Proposition~\ref{p:1}(b)) we have
\[
 P_{[2m,2n)}(z)=P_{[m,n)}(z^2)+z P_{[m,n)}(-z^2),
\]
and hence
\[
 P_{[2m,2n)}(-z)=P_{[m,n)}(z^2)-z P_{[m,n)}(-z^2).
\]
Thus by the paralellogram rule
\begin{align*}
 |P_{[2m,2n)}(z)|^2+|P_{[2m,2n)}(-z)|^2=2\big(|P_{[m,n)}(z^2)|^2+|P_{[m,n)}(-z^2)|^2\big).
\end{align*}
The first statement follows on taking supremums over $|z|=1$.
The second statement then follows by taking $m=0$.
\end{proof}

As an example, we see that
\begin{equation}\label{e:ptnorm}
 \|P_t\|_L=\|P_{<2^t}\|_L=2^{t/2}\|P_{<1}\|_L=2^{t/2}\|1\|_L=2^{t/2}\cdot\sqrt{2}=2^{(t+1)/2}.
\end{equation}
Clearly $\|\pm z^tP(\pm z^s)\|_L=\|P(z)\|_L$ for any $s\ne0$, so Proposition~\ref{p:1}(c) implies that
\begin{equation}\label{e:qtnorm}
 \|Q_t\|_L=\|\pm z^{2^t-1}P_t(-z^{-1})\|_L=\|P_t\|_L=2^{(t+1)/2}.
\end{equation}

As we clearly have $\|P\|_\infty\le\|P\|_L\le \sqrt{2}\|P\|_\infty$, we deduce from Lemma~\ref{l:scale}
that in general
\[
 \limsup_{k\to\infty}\frac{\|P_{[2^km,2^kn)}\|_\infty}{2^{k/2}} \le \|P_{[m,n)}\|_L,
\]
and a natural question is how much do these quantities differ. Indeed,
they are equal.

\begin{theorem}\label{t:dense}
\[
 \lim_{k\to\infty}\frac{\|P_{[2^km,2^kn)}\|_\infty}{2^{k/2}} = \|P_{[m,n)}\|_L.
\]
\end{theorem}
We defer the proof of this result (which is not needed in the proofs of Theorems \ref{t:onedim} or~\ref{t:twodim})
to Section~\ref{s:dense}.

Finally we note that it is easy to see that there exists a constant $C>0$ such that
\begin{equation}\label{e:cbound}
 \|P_{[m,n)}\|_L\le C\sqrt{n-m}
\end{equation}
for all $n\ge m\ge0$. Indeed, we may assume $n>m$ and pick a maximal $k$ such that $m\le 2^kr\le n$ for some
(necessarily odd) $r\in\N$. As $n<2^k(r+1)=2^{k+1}\frac{r+1}{2}$ we can write $n=2^kr+2^{t_1}+\dots+2^{t_p}$ with
$k>t_1>t_2>\dots>t_p$. Now note that, by
Proposition~\ref{p:1}(b), $P_{[2^kr,n)}$ can be decomposed into blocks of length $2^{t_i}$ each of which is (up
to multiplication by a power of~$z$) either $\pm P_{t_i}$ or $\pm Q_{t_i}$.
Thus by \eqref{e:ptnorm} and \eqref{e:qtnorm} we have
$\|P_{[2^kr,n)}\|_L\le\sum_i 2^{(t_i+1)/2}=O(2^{t_1/2})=O(\sqrt{n-m})$.
Similarly writing $m=2^kr-2^{s_1}-\dots-2^{s_q}$, $k>s_1>s_2>\dots>s_q$, we see that $P_{[m,2^kr)}$ can be decomposed
into blocks $\pm P_{s_i}$ or $\pm Q_{s_i}$ and $\|P_{[m,2^kr)}\|_L\le\sum_i 2^{(s_i+1)/2}=O(2^{s_1/2})=O(\sqrt{n-m})$.
The result then follows as $\|P_{[m,n)}\|_L\le \|P_{[m,2^kr)}\|_L+\|P_{[2^kr,n)}\|_L$.

\section{Proof of Theorem~\ref{t:onedim}}\label{s:onedim}

Define the function $f$ by
\[
 f(n)=\|P_{<n}\|_L^2
\]
for $n\in\N=\{0,1,2,\dots\}$. Lemma~\ref{l:scale} implies that $f(2n)=2f(n)$,
and so allows us to consistently extend this definition to
all non-negative dyadic rationals $x=\frac{n}{2^k}$ by defining
\begin{equation}\label{e:scale}
 f(x)=2^{-k}f(2^kx).
\end{equation}
Now the triangle inequality, the observation that $P_{<n}(z)=P_{<m}(z)+P_{[m,n)}(z)$,
and~\eqref{e:cbound}, imply that
\[
 |f(n)^{1/2}-f(m)^{1/2}|\le C\sqrt{n-m}.
\]
By \eqref{e:scale} this implies
\begin{equation}\label{e:fcont}
 |f(x)^{1/2}-f(y)^{1/2}|\le C\sqrt{y-x}
\end{equation}
for any dyadic rationals $y\ge x\ge0$, and hence
$f$ can be extended by continuity to a continuous function $f\colon[0,\infty)\to\R$
which satisfies
\begin{equation}\label{e:fdouble}
 f(2x)=2f(x)
\end{equation}
for all $x\ge0$.

A more refined version of the continuity statement \eqref{e:fcont}
can be given if $y$ is sufficiently close to a simple dyadic rational~$x$.

\begin{lemma}\label{l:fcont}
 If\/ $2^kx\in\N$ then
 \[
  |f(y)^{1/2}-f(x)^{1/2}|\le f(|y-x|)^{1/2}
 \]
 for all\/ $y\ge0$ with\/ $|y-x|\le 2^{-k-1}$.
\end{lemma}
\begin{proof}
It is enough by continuity to prove this for any dyadic rational~$y$, so pick a $t\in\N$ such that
$2^{k+t}y$ is an integer. Writing $n=2^kx$ and $r=2^{k+t}|y-x|$, we
have $r\le 2^{t-1}$ and $2^{k+t}y=2^tn\pm r$. Now
\[
 P_{<2^tn+r}(z)=P_{<2^tn}(z)\pm z^{2^tn}P_{<r}(z)
\]
and also
\[
 P_{<2^tn-r}(z)=P_{<2^tn}(z)\pm z^{2^tn-1}P_{<r}(-z^{-1}).
\]
Indeed, these follow from Proposition~\ref{p:1}(b) as $P_\infty(z)$
can be decomposed into blocks of the form $\pm z^{2^tm}P_t(z)$ when $m$ is even
and $\pm z^{2^tm}Q_t(z)$ when $m$ is odd. The first equality then
follows as the first $r\le 2^{t-1}$ terms of either $P_t$ or $Q_t$ forms a $P_{<r}$.
The second equality follows from Proposition~\ref{p:1}(c) which implies the
last $r$ terms of $P_t$ or $Q_t$ forms a $\pm z^{2^t-1}P_{<r}(-z^{-1})$.

The triangle inequality now implies that
\[
 \big|\|P_{<2^{k+t}y}\|_L-\|P_{<2^{k+t}x}\|_L\big|\le \|P_{<2^{k+t}|y-x|}\|_L,
\]
from which we deduce from \eqref{e:scale} that $|f(y)^{1/2}-f(x)^{1/2}|\le f(|x-y|)^{1/2}$.
\end{proof}

\begin{figure}
\centerline{\includegraphics[width=2.5in]{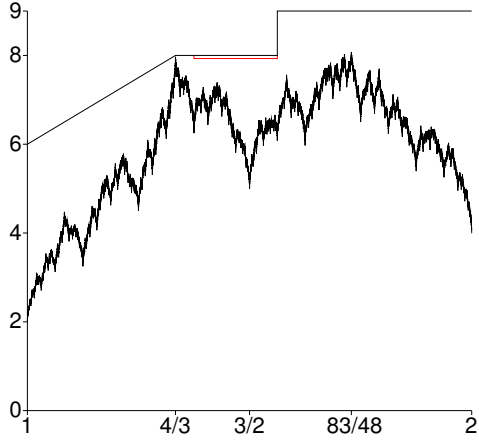}}
\caption{Graph of $f(x)$ with bounds proven in Theorem~\ref{t:fbound} (black) and \eqref{e:fstrong} (red).
Note that $f(\frac{83}{48})>8$, so we are unable to prove a bound $f(x)\le 8$
on the whole interval $[\frac{4}{3},2]$.}\label{f:1}
\end{figure}

We now prove a slightly weaker version of Theorem~\ref{t:onedim}, which is nevertheless enough
to imply $\|P_{<n}\|_\infty\le \sqrt{6n}$.

\begin{theorem}\label{t:fbound}
 We have the bounds
 \begin{equation}\label{e:fbound}
  f(x)\le
  \begin{cases}
   6x,&\text{if }x\in[1,\frac{4}{3}];\\
   8,&\text{if }x\in[\frac{4}{3},\frac{25}{16}];\\
   9,&\text{if }x\in[\frac{25}{16},2].
  \end{cases}
 \end{equation}
 In particular, $f(x)\le 6x$ for all $x\ge0$.
\end{theorem}
\begin{proof}
It is enough by continuity to prove these inequalities for a dyadic rational,
and hence it is enough to prove the appropriately scaled inequalities
for integers $n=2^kx$. We prove the result by induction on~$n$.
Clearly $f(0)=\|0\|_L^2=0$ and $f(1)=\|1\|_L^2=2$ satisfy these conditions.

First suppose $1\le x\le\frac{4}{3}$. Write $x=1+y$ so that $n=2^kx=2^k+r$, $r=2^ky<2^{k-1}$.
Then by induction $f(r)\le 6r$ and so $f(y)\le 6y$. Now, by Lemma~\ref{l:fcont},
\[
 f(x)\le \big(f(1)^{1/2}+f(y)^{1/2}\big)^2 \le \big(\sqrt{2}+\sqrt{6y}\big)^2=2+4\sqrt{3y}+6y\le 6+6y=6x
\]
for all $y\le\frac{1}{3}$.

Now suppose $\frac{4}{3}\le x\le \frac{11}{8}$. Again write $x=1+y$ so that $n=2^k+r$,
$r=2^ky<2^{k-1}$. Now $4y\in [\frac{4}{3},\frac{3}{2}]$, so by induction ($r<n$),
$f(y)=\frac{1}{4}f(4y)\le\frac{1}{4}\cdot 8=2$.
Hence
\[
 f(x)\le \big(f(1)^{1/2}+f(y)^{1/2}\big)^2\le\big(\sqrt{2}+\sqrt{2}\big)^2=8,
\]
as required.

\begin{table}
\[\begin{array}{llll}
x\text{ (binary)}&x\text{ (decimal)}&f(x)&\text{Interval covered}\\
1.011     &1.375000&6.250000&[1.358355,1.391645]^1\\
1.01101   &1.406250&6.491173&[1.390625,1.421875]^*\\
1.011011  &1.421875&6.955324&[1.415772,1.427978]^2\\
1.0111    &1.437500&6.625000&[1.427730,1.447270]^1\\
1.1       &1.500000&5.000000&[1.437500,1.562500]^3\\
\end{array}\]
\caption{Values of $f(x)$ used to bound $f(x)\le 7.92$ in $[1.375,1.5625]=[\frac{11}{8},\frac{25}{16}]$
along with the intervals where bound is proven. The index $i$ on the interval indicates
that the range $[x-r,x+r]$ was limited in this case by a bound on $f(r)$ corresponding to a scaled version
of case $i$ in~\eqref{e:fbound}. A star on the interval indicates $r$ was limited by
the restriction $|y-x|\le r=2^{-k-1}$ in Lemma~\ref{l:fcont}.}\label{t:1}
\end{table}

\begin{table}
\[\begin{array}{llll}
x\text{ (binary)}&x\text{ (decimal)}&f(x)&\text{Interval covered}\\
1.101     &1.625000&5.971801&[1.562500,1.687500]^*\\
1.1011    &1.687500&7.090947&[1.668559,1.706441]^1\\
1.10111   &1.718750&7.284252&[1.703125,1.734375]^*\\
1.11      &1.750000&6.500000&[1.716177,1.783823]^1\\
1.1101    &1.812500&6.239011&[1.781250,1.843750]^*\\
10.       &2.000000&4.000000&[1.833334,2.166666]^1\\
\end{array}\]
\caption{Values of $f(x)$ used to bound $f(x)\le 9$ in $[1.5625,2]=[\frac{25}{16},2]$
along with the intervals where bound is proven.
The indices $i$ on the intervals are as in Table~\ref{t:1}.}\label{t:2}
\end{table}

It remains to prove the theorem in the case when $x\in[\frac{11}{8},2]$, the last statement then following from
the fact that $f(x)\le 6x$ for all $x\in[1,2]$, and $f(2x)=2f(x)$ for all $x\ge0$. In fact, it will help
in the proof of Theorem~\ref{t:onedim} to prove the very slightly stronger bound
\begin{equation}\label{e:fstrong}
 f(x)\le 7.92\qquad\text{if }x\in[\tfrac{11}{8},\tfrac{25}{16}].
\end{equation}
The inequalities \eqref{e:fbound} and \eqref{e:fstrong} however are never equalities on $[\frac{11}{8},2]$
(see Figure~\ref{f:1} for a plot of $f(x)$).
As $f(x)$ can be readily calculated by computer, Lemma~\ref{l:fcont}
allows us to provide a computer assisted proof on an interval around any dyadic point. The result will
then follow by exhibiting a collection of such intervals that cover $[\frac{11}{8},2]$.

More specifically, we use the values of $x$ in Table~\ref{t:1} to show that $f(y)\le 7.92$
for all $y\in[\frac{11}{8},\frac{25}{16}]$, and the values of $x$ in Table~\ref{t:2} to show that $f(y)\le 9$
for all $y\in[\frac{25}{16},2]$. In each case we use Lemma~\ref{l:fcont} to bound $f(y)$
in an interval $[x-r,x+r]$ around $x$ using induction and \eqref{e:fbound} (scaled appropriately using
$f(|y-x|)=2^{-t}f(2^t|y-x|)$ with $2^t|y-x|\in[1,2]$) to bound $f(|y-x|)$ for $|y-x|\le r$.

Computer calculations of $f(x)$ were performed by evaluating $P_{<n}(z)$ for $n=2^kx$ on all $2^{24}$th
roots of unity. The maximum error bound in $f(x)$ being easily seen to be less that
$10^{-6}$ in all cases (e.g., by the argument on page 551 of~\cite{M17}).
\end{proof}

\begin{proof}[Proof of Theorem~\ref{t:onedim}]
Write $n_k=\frac{4}{3}\cdot 2^k+\frac{1}{3}$ and note that $n_k$ is only an integer
when $k$ is odd, and that $\sqrt{6n_k-2}-1=2^{(k+3)/2}-1$.

We shall prove by induction on $k$ that
\begin{align}
 \|P_{<n}\|_\infty&\le \sqrt{6n-2}-1,&\rlap{for $1\le n\le 2^{k+1}$; and}\qquad\qquad\qquad\qquad\qquad\label{e:nbound1}\\
 \|P_{<n}\|_\infty&\le 2^{(k+3)/2}-1,&\rlap{for $n_k\le n\le\tfrac{25}{16}\cdot 2^k$;}\qquad\qquad\qquad\qquad\qquad\label{e:nbound2}
\end{align}
where we note that \eqref{e:nbound2} implies \eqref{e:nbound1} for $n_k\le n\le\tfrac{25}{16}\cdot 2^k$.
It is easy to see that \eqref{e:nbound1} and \eqref{e:nbound2} hold for $k=0,1$, so assume $k\ge2$.

Suppose $2^k<n<n_k$ and write $n=2^k+r$, where $r<n_k-2^k=n_{k-2}$.
As $0<r\le 2^{k-1}$ we have $P_{<n}(z)=P_k(z)+z^{2^k}P_{<r}(z)$ and, by induction,
\[
 \|P_{<n}\|_\infty\le \|P_k\|_\infty+\|P_{<r}\|_\infty
  \le 2^{(k+1)/2}+\sqrt{6r-2}-1.
\]
Now $r\le n_{k-2}$ implies $\sqrt{6r-2}\le 2^{(k-1)/2}$, so
\begin{align*}
 (2^{(k+1)/2}+\sqrt{6r-2})^2
 &=2^{k+1}+6r-2+2^{(k+3)/2}\sqrt{6r-2}\\
 &\le 2^{k+1}+6r-2 + 2^{(k+3)/2}\cdot 2^{(k-1)/2}\\
 &= 6(2^k+r)-2=6n-2.
\end{align*}
Hence $\|P_{<n}\|_\infty\le \sqrt{6n-2}-1$, as required.

Now suppose $n_k\le n<\frac{11}{8}\cdot 2^k$ and write $n=2^k+r$ with
$n_{k-2}\le r\le\frac{3}{2}\cdot 2^{k-2}<\frac{25}{16}\cdot 2^{k-2}$.
Again we have $P_{<n}(z)=P_k(z)+z^{2^k}P_{<r}(z)$ and, by induction,
\[
 \|P_{<n}\|_\infty\le \|P_k\|_\infty+\|P_{<r}\|_\infty\le 2^{(k+1)/2}+2^{(k+1)/2}-1=2^{(k+3)/2}-1,
\]
as required.

Now for $\frac{11}{8}\cdot 2^k\le n\le \frac{25}{16}\cdot 2^k$ we simply use~\eqref{e:fstrong} to obtain
\[
 \|P_{<n}\|_\infty\le \|P_{<n}\|_L\le \sqrt{7.92}\cdot 2^{k/2}<2^{(k+3)/2}-1,
\]
where the last inequality holds for $k\ge 13$. For $k\le 12$
computer calculations show directly that $\|P_{<n}\|_L<2^{(k+3)/2}-1$ for this range of~$n$.

Finally, for $\frac{25}{16}\cdot 2^k\le n\le 2^{k+1}$ we have
\[
 \|P_{<n}\|_\infty\le \|P_{<n}\|_L\le 3\cdot2^{k/2}< \sqrt{6\cdot \tfrac{25}{16}\cdot 2^k-2}-1\le \sqrt{6n-2}-1
\]
for $k\ge7$, and for $k\le 6$ computer calculations show directly that $\|P_{<n}\|_L<\sqrt{6n-2}-1$
for this range of~$n$.
\end{proof}

\section{Proof of Theorem~\ref{t:twodim}}\label{s:twodim}

We can define, in analogy to $f(x)$ above, the function
\[
 f(m,n):=\|P_{[m,n)}\|_L^2
\]
and extend by Lemma~\ref{l:scale} and then by continuity to a continuous function $f(x,y)$ defined
for all $0\le x\le y$, $x,y\in\R$, that satisfies
\[
 f(2x,2y)=2f(x,y).
\]

Again the strategy is to use a computer to check most of the parameter space $(x,y)$, which by scaling
and translating can be assumed to be $[0,2]\times[2,4]$. The main difficulty is that the $P_{[m,n)}$
corresponding to $(x,y)$ near the extremal point $(\frac{5}{3},\frac{8}{3})$ does not exhibit such a simple
decomposition as before. Thus we will need to deal with a more complicated version of our $\|\cdot\|_L$ norm.

Define the following function for any $r,s\in\N$,
\[
 g(r,s)=\sup_{|\alpha|=1}\|P_{<s}(z)+\alpha z^{-1} P_{<r}(-z^{-1})\|_L^2.
\]

\begin{proposition}\label{p:gbasic}
 The function\/ $g$ satisfies the following properties.
 \begin{itemize}
 \item[$($a$)$] For all\/ $r,s\ge0$, $g(s,r)=g(r,s)$.
 \item[$($b$)$] For all\/ $r,s\ge0$, $g(2r,2s)=2g(r,s)$.
 \item[$($c$)$] There exists a constant\/ $C>0$ such that for all\/ $r,s,r',s'\ge0$,
 \[
  \big|g(r,s)^{1/2}-g(r',s')^{1/2}\big|\le C|r-r'|^{1/2}+C|s-s'|^{1/2}.
 \]
 \end{itemize}
\end{proposition}
\begin{proof}
The first part follows immediately by simply substituting $z\mapsto -z^{-1}$ and $\alpha\mapsto -\alpha^{-1}$
in the definition of $g(s,r)$. For the second part we note by Proposition~\ref{p:1}(b) that
\[
 P_{<2s}(z)+\alpha z^{-1} P_{<2r}(-z^{-1})=P_{<s}(z^2)+zP_{<s}(-z^2)+\alpha z^{-1} P_{<r}(z^{-2})-\alpha z^{-2}P_{<r}(-z^{-2}),
\]
and hence
\[
 P_{<2s}(-z)-\alpha z^{-1} P_{<2r}(z^{-1})=P_{<s}(z^2)-zP_{<s}(-z^2)-\alpha z^{-1} P_{<r}(z^{-2})-\alpha z^{-2}P_{<r}(-z^{-2}).
\]
Thus by the parallelogram rule
\begin{align*}
 \big|P_{<2s}(z)+\alpha z^{-1} &P_{<2r}(-z^{-1})\big|^2+\big|P_{<2s}(-z)-\alpha z^{-1} P_{<2r}(z^{-1})\big|^2\\
 &=2\big|P_{<s}(z^2)-\alpha z^{-2}P_{<r}(-z^{-2})\big|^2+2\big|zP_{<s}(-z^2)+\alpha z^{-1} P_{<r}(z^{-2})\big|^2\\
 &=2\big|P_{<s}(z^2)-\alpha z^{-2}P_{<r}(-z^{-2})\big|^2+2\big|P_{<s}(-z^2)+\alpha z^{-2} P_{<r}(z^{-2})\big|^2\\
 &\le2\|P_{<s}(z^2)-\alpha z^{-2}P_{<r}(-z^{-2})\|_L^2\\
 &\le 2g(r,s).
\end{align*}
The result now follows by taking the supremum over $z$ and~$\alpha$.

The last statement is immediate from the triangle inequality for $\|\cdot\|_L$ together with~\eqref{e:cbound}.
\end{proof}

As with the function~$f$, we can now extend the definition of $g$ to non-negative dyadic rationals by setting
\begin{equation}\label{e:gscale}
 g(x,y)=2^{-k}g(2^kx,2^ky),
\end{equation}
where $2^kx,2^ky\in\N$, and then extend the definition of $g$ by continuity (Proposition~\ref{p:gbasic}(c))
to all real $x,y\ge0$. The following shows that we can use the function $g$ to bound the function $f$ in a (rather large)
neighborhood of the critical point $(\frac{5}{3},\frac{8}{3})$.

\begin{lemma}\label{l:compfg}
 If\/ $0\le x,y\le1$ then
 \[
  f(2-x,2+y)\le g(x,y).
 \]
\end{lemma}
\begin{proof}
By continuity it is enough to prove this for dyadic rationals, and by scaling it is then
enough to show that for integers $r=2^{k-1}x$ and $s=2^{k-1}y$ with $0\le r,s\le 2^{k-1}$ we have
\[
 f(2^k-r,2^k+s)\le g(r,s).
\]
this however follows immediately from the definitions of $f$ and $g$ together with the fact that
$P_{[2^k-r,2^k+s)}(z)=P_{[2^k-r,2^k)}(z)+P_{[2^k,2^k+s)}(z)=\pm z^{2^k-1}P_{<r}(-z^{-1})+z^{2^k}P_{<s}(z)$.
\end{proof}

\begin{remark}
The difference between $g(r,s)$ and $f(2^k-r,2^k+s)$ is that we lose information
on the phase difference of the $P_{<r}(-z^{-1})$ term and the $P_{<s}(z)$ term.
This is important due to the rather strange way in which we will need to decompose our
polynomials $P_{[m,n)}$ when $(m,n)$ is close to $(m_k,n_k)$.
However, it is \emph{not\/} enough to define $g(r,s)$ more simply
as $\||P_{<s}(z)|+|P_{<r}(-z^{-1})|\|_L^2$
as this quantity is too large near the critical values of $(r,s)$. It is important
that the \emph{same\/} $\alpha$ is used for both the $z$ and $-z$ terms defining
the $\|\cdot\|_L$ norm in the definition of~$g$.
\end{remark}

Although the definition of $g(r,s)$ is easy to use in proofs, it does not
look so easy to calculate numerically due to the fact that we are taking
supremums over both $z$ and~$\alpha$. However, one can rewrite $g(r,s)$
in a form that avoids the supremum over~$\alpha$. The following was therefore used
in the numerical calculations of $g(r,s)$.

\begin{lemma} For non-negative integers\/ $r$ and\/ $s$,
 \begin{gather*}
  g(r,s)=\sup_{|z|=1}\big\{|P_{<r}(z)|^2+|P_{<r}(-z)|^2+|P_{<s}(z)|^2+|P_{<s}(-z)|^2\\
  \qquad\qquad+2|P_{<s}(z)P_{<r}(-z)-P_{<s}(-z)P_{<r}(z)|\big\}.
 \end{gather*}
\end{lemma}
\begin{proof}
Write the function $g(r,s)$ as $\sup_{|\alpha|=1}\sup_{|z|=1} S_{r,s}(\alpha,z)$, where
\begin{align*}
 S_{r,s}(\alpha,z)
 &=|P_{<s}(z)+\alpha z^{-1} P_{<r}(-z^{-1})|^2+|P_{<s}(-z)-\alpha z^{-1} P_{<r}(z^{-1})|^2\\
 &=\big(P_{<s}(z)+\alpha \bar z P_{<r}(-\bar z)\big)\big(P_{<s}(\bar z)+\bar\alpha z P_{<r}(-z)\big) + \{z\mapsto-z\}\\
 &=|P_{<s}(z)|^2+|P_{<r}(-z)|^2+\alpha \bar zP_{<r}(-\bar z)P_{<s}(\bar z)+\bar\alpha zP_{<s}(z)P_{<r}(-z) + \{z\mapsto-z\}\\
 &=|P_{<r}(z)|^2+|P_{<r}(-z)|^2+|P_{<s}(z)|^2+|P_{<s}(-z)|^2\\
 &\quad +\bar \alpha z\big(P_{<s}(z)P_{<r}(-z)-P_{<s}(-z)P_{<r}(z))+\{\text{cplx. conj.}\}.
\end{align*}
Clearly the sum of the last two terms in maximized when $\alpha$ is chosen so that
\[
 \bar\alpha z\big(P_{<s}(z)P_{<r}(-z)-P_{<s}(-z)P_{<r}(z))
\]
is a positive real. Thus
\begin{align*}
 \sup_{|\alpha|=1} S_{r,s}(\alpha,z)&=|P_{<r}(z)|^2+|P_{<r}(-z)|^2+|P_{<s}(z)|^2+|P_{<s}(-z)|^2\\
 &\quad +2|P_{<s}(z)P_{<r}(-z)-P_{<s}(-z)P_{<r}(z)|.
\end{align*}
The result follows on taking the supremum over~$z$.
\end{proof}

The following are refined versions of the continuity statements for $f$ and $g$ that we will need later
in the computer assisted proofs.

\begin{lemma}\label{l:fgcont}
 If\/ $2^kx,2^ky\in\N$ and\/ $|x-x'|,|y-y'|\le 2^{-k-1}$, then
 \begin{align}
  \big|f(x',y')^{1/2}-f(x,y)^{1/2}\big|&\le f(|x'-x|)^{1/2}+f(|y'-y|)^{1/2}\le 3\cdot 2^{-k/2},\label{e:fcont2}\\
  \big|g(x',y')^{1/2}-g(x,y)^{1/2}\big|&\le f(|x'-x|)^{1/2}+f(|y'-y|)^{1/2}\le 3\cdot 2^{-k/2}.\label{e:gcont}
 \end{align}
\end{lemma}
\begin{proof}
Follows from the same proof as in Lemma~\ref{l:fcont}. For the last inequality we note that
if $z\le 2^{-k-1}$ then $f(z)\le 2^{-k-2}f(2^{k+2}z)\le 9\cdot 2^{-k-2}$. Hence
$f(|x'-x|)^{1/2}+f(|y'-y|)^{1/2}\le 2\cdot 3\cdot 2^{-(k+2)/2}=3\cdot 2^{-k/2}$.
\end{proof}

The following is the key inequality needed to bound $g(x,y)$ near the critical point $(\frac{4}{3},\frac{8}{3})$.

\begin{lemma}\label{l:gcrit}
 For all\/ $x\in[0,\frac12]$ and\/ $y\in[0,1]$,
 \[
  g(1+x,2+y)^{1/2}\le \sqrt{10}+g(x,y)^{1/2}.
 \]
\end{lemma}
\begin{proof}
Writing $r=2^{k-1}x$ and $s=2^{k-1}y$ we have $0\le r\le 2^{k-2}$ and $0\le s\le 2^{k-1}$. Thus
$P_{<2^{k-1}+r}(z)=P_{<2^{k-1}}(z)+z^{2^{k-1}}P_{<r}(z)$ and $P_{<2^k+s}(z)=P_{2^k}(z)+z^{2^k}P_{<s}(z)$.
Clearly we may assume $k\ge2$, so that
\begin{align*}
 g(2^{k-1}&+r,2^k+s)^{1/2}\\
 &=\sup_\alpha\big\|P_{<2^k}(z)+z^{2^k}P_{<s}(z)+\alpha z^{-1} P_{<2^{k-1}}(-z^{-1})+\alpha z^{-1-2^{k-1}}P_{<r}(-z^{-1})\big\|_L\\
 &\le \sup_\alpha\big\|P_{<2^k}(z)+\alpha z^{-1} P_{<2^{k-1}}(-z^{-1})\|_L+\sup_\alpha\|z^{2^k}P_{<s}(z)+\alpha z^{-1-2^{k-1}}P_{<r}(-z^{-1})\big\|_L\\
 &=g(2^{k-1},2^k)^{1/2}+\sup_{\beta,z}\big\|\big(P_{<s}(z)+\beta z^{-1}P_{<r}(-z^{-1}),P_{<s}(-z)-\beta z^{-1}P_{<r}(z^{-1})\big)\big\|_2\\
 &=g(2^{k-1},2^k)^{1/2}+\sup_\beta\|P_{<s}(z)+\beta z^{-1}P_{<r}(-z^{-1})\|_L\\
 &=g(2^{k-1},2^k)^{1/2}+g(r,s)^{1/2},
\end{align*}
where $\beta=\alpha z^{-3\cdot2^{k-1}}=\alpha(-z)^{-3\cdot 2^{k-1}}$.
Hence after scaling we have
\[
 g(1+x,2+y)^{1/2}\le g(1,2)^{1/2}+g(x,y)^{1/2}.
\]
Finally we observe that
\begin{align*}
 g(1,2)&=\sup_{\alpha,z}\big(|1+z+\alpha z^{-1}|^2+|1-z-\alpha z^{-1}|^2\big)\\
 &=\sup_{\alpha,z}2\big(|1|^2+|z+\alpha z^{-1}|^2)=2\cdot (1^2+2^2)=10.\qedhere
\end{align*}
\end{proof}

We now come to the key bound we need on $g(x,y)$.

\begin{lemma}\label{l:gbound}
 For\/ $x\in[0,2]$, $y\in[0,4]$ we have
 \begin{equation}\label{e:gb}
  g(x,y)\le \min\{10(x+y),40\}.
 \end{equation}
 In particular, $g(x,y)\le 10(x+y)$ for all\/ $x,y\ge0$.
\end{lemma}
\begin{proof}
We use induction to bound $g(r,s)$ for integers $r,s$, and scale using \eqref{e:gscale},
however for ease of exposition we will write the proof in terms of $x,y\in\R$, which
by continuity may be considered dyadic rationals.

Firstly, we can reduce to the case when $(x,y)$ lies inside the blue contour in Figure~\ref{f:2},
namely
\[
 (x,y)\in B:=([0,2]\times[0,4])\setminus ([0,1]\times[0,2])\setminus ([0,2]\times[0,1]).
\]
Indeed, for $x\ge y$ we have $g(x,y)=g(y,x)$, and for $x\le y$, $(x,y)\in[0,1]\times[0,2]$, we have
$2^k(x,y)\in B$ for some $k\ge1$. Then
\[
  g(x,y)=\tfrac{1}{2^k}g(2^kx,2^ky)\le \min\{10(x+y),40\cdot 2^{-k}\}\le\min\{10(x+y),40\}.
\]

If $x=1+x'\in[1,\frac{3}{2}]$ and $y=2+y'\in[2,3]$ (inside red contour in Figure~\ref{f:1})
induction ($2^kx'<2^kx$, $2^ky'<2^ky$) implies that
\[
 g(x',y')=\tfrac14 g(4x',4y')\le\tfrac14 \min\{10(4x'+4y'),40\}=\min\{10(x'+y'),10\}.
\]
Thus by Lemma~\ref{l:gcrit}
\[
 g(x,y)^{1/2}\le \sqrt{10}+\min\{\sqrt{10(x'+y')},\sqrt{10}\},
\]
so for $x+y\le 4$ ($x'+y'\le 1$) we have
\begin{align*}
 g(x,y)&\le \big(\sqrt{10}+\sqrt{10(x'+y')}\big)^2=10+20\sqrt{x'+y'}+10(x'+y')\\
 &\le 30+10(x'+y')=10(x+y),
\end{align*}
and for $x+y\ge 4$ ($x'+y'\ge1$)
\[
 g(x,y)\le \big(\sqrt{10}+\sqrt{10}\big)^2=40.
\]

In the remaining cases (between the red and blue contours) the inequality is strict, and so can be proved by computer.
We divide up the region into dyadic squares $S_{r,s}=[2^{-k}r,2^{-k}(r+1)]\times[2^{-k}s,2^{-k}(s+1)]$ and
evaluate $g(x,y)$ numerically at each corner of $S_{r,s}$. We then divide $S_{r,s}$ up into 4 smaller dyadic squares
$S'_{r',s'}=[2^{-k-1}r',2^{-k-1}(r'+1)]\times[2^{-k-1}s',2^{-k-1}(s'+1)]$, $r'\in\{2r,2r+1\}$, $s'\in\{2s,2s+1\}$.
Using Lemma~\ref{l:fgcont} we try to prove the bound \eqref{e:gb} for each of the four smaller squares $S'_{r',s'}$
using the value of $g(x,y)$ at the corresponding corner of~$S_{r,s}$. Note that all points $(x',y')\in S'_{r',s'}$
satisfy the conditions $|x'-x|,|y'-x|\le 2^{-k-1}$. If this fails, we recursively subdivide each $S'_{r',s'}$
in the same way. Once we get to squares of side length $2^{-6}$ we give up and mark the square as bad.

This procedure was applied to the whole of $[0,4]^2$ and the result is shown in Figure~\ref{f:2}. The only bad squares
that lie inside the blue contour also lie inside the red contour, so we are done.
\end{proof}

\begin{figure}
\centerline{\includegraphics[width=2.5in]{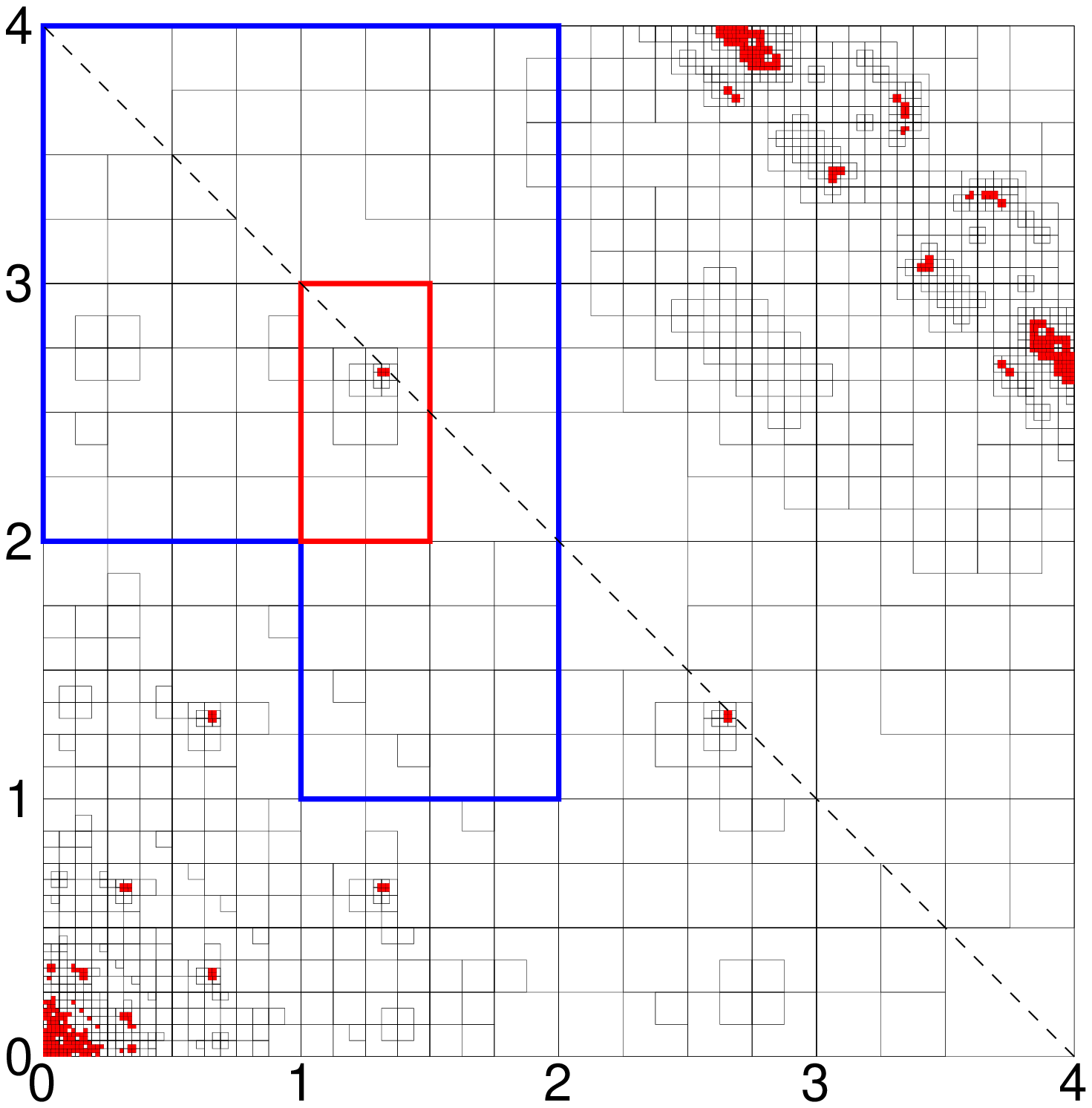}\qquad\includegraphics[width=2.5in]{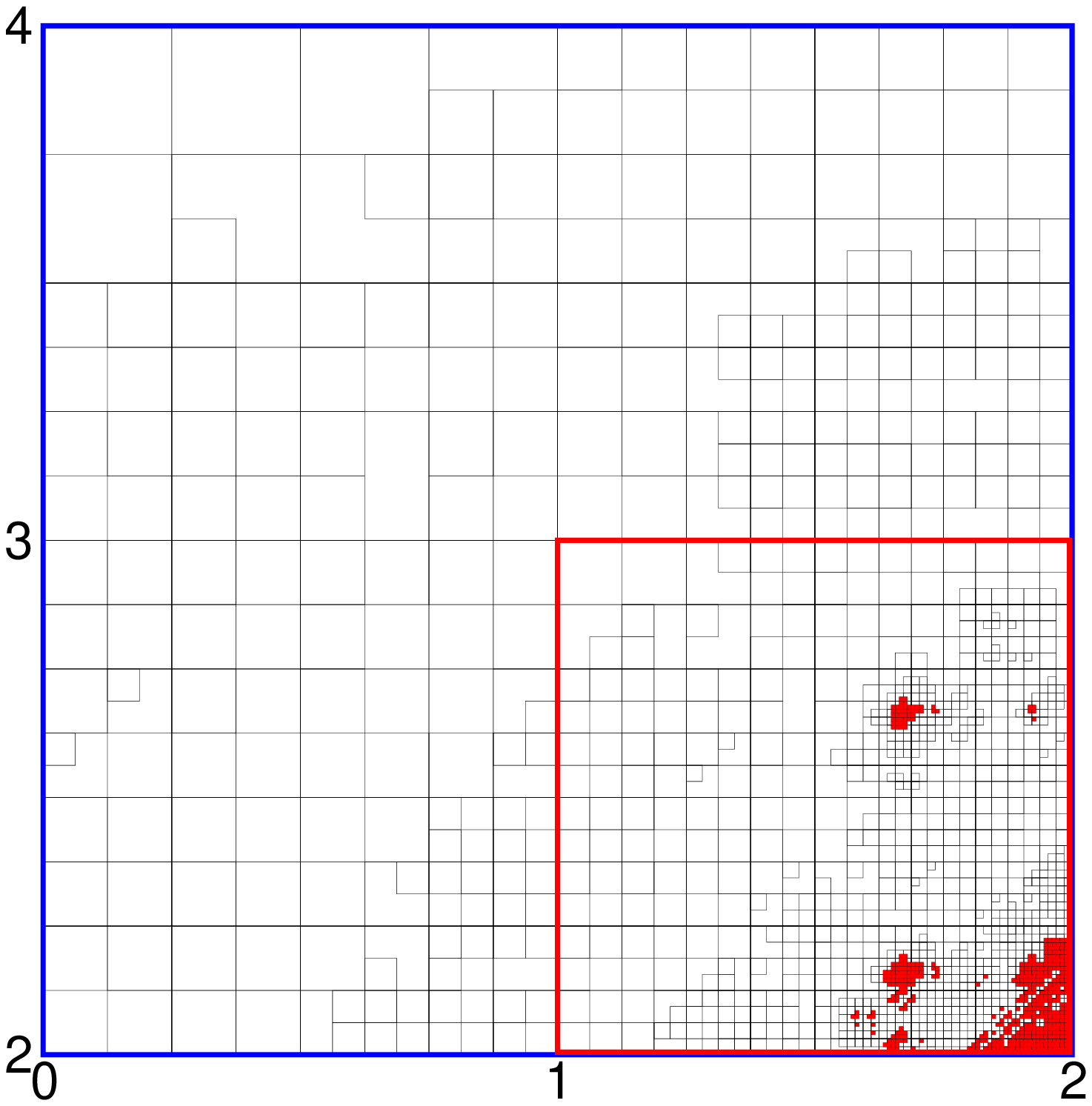}}
\caption{Bounding $g(x,y)$ in Lemma~\ref{l:gbound} (left) and $f(x,y)$ in Theorem~\ref{t:twodim} (right).
Region between red and blue contours is divided recursively into squares in an attempt to prove bounds. The boundaries
of the squares are shown only if the first attempt to bound $f$ or $g$ on them failed (so they occur as an $S_{r,s}$
square in the proofs).
On the left, the dashed line indicates the line $10(x+y)=40$. The regions outside of the blue contours and inside the red contours
are shown for illustration only and are not used in the proofs.}\label{f:2}
\end{figure}

\begin{proof}[Proof of Theorem~\ref{t:twodim}.]
The result clearly holds for $0\le m\le n<2$, so assume $n\ge2$ and fix $k\ge0$ so that $2^{k+1}\le n<2^{k+2}$.
Now if $m\ge 2^{k+1}$ we can use the fact that $Q_{k+1}(z)=\pm z^{2^{k+1}-1}P_{k+1}(-z^{-1})$
to deduce that $\|P_{[m,n)}\|_L=\|P_{[2^{k+2}-n,2^{k+2}-m)}\|_L$. But
$2^{k+2}-m<n$, so we are done by induction on~$n$.
Thus we may assume that $m=2^kx$, $n=2^ky$, with $(x,y)\in [0,2]\times[2,4]$.
For $(x,y)\in[1,2]\times[2,3]$ we use Lemmas \ref{l:compfg} and \ref{l:gbound} to deduce that
\[
 f(x,y)\le g(2-x,y-2)\le 10(y-x).
\]
In all other cases the bounds are strict, and so can be proved by computer in an exactly analogous way
to Lemma~\ref{l:gbound}.
\end{proof}

\section{Proof of Theorem~\ref{t:dense}}\label{s:dense}

We first describe the strategy used to prove Theorem~\ref{t:dense}.
We apply Proposition~\ref{p:1}(b) to deduce that
\begin{align*}
 P_{[2^{k+1}m,2^{k+1}n)}(w)
 &=P_k(w)P_{[m,n)}(w^{2^{k+1}})+w^{2^k}Q_k(w)P_{[m,n)}(-w^{2^{k+1}})\\
 &=\big\langle (P_k(w),Q_k(w))^T, (P_{[m,n)}(\bar z^2),\bar z P_{[m,n)}(-\bar z^2))^T\big\rangle,
\end{align*}
where $z=w^{2^k}$ and $\langle u,v\rangle = u^T\bar v$ is the standard inner product on $\C^2$.
To maximize this expression we pick $z$ to be such that $\|(P_{[m,n)}(z^2),P_{[m,n)}(-z^2))\|_2$
is maximized, so that then
$\|(P_{[m,n)}(\bar z^2),\bar zP_{[m,n)}(-\bar z^2))\|_2=\|(P_{[m,n)}(z^2),P_{[m,n)}(-z^2))\|_2=\|P_{[m,n)}\|_L$.
We then wish
to pick $w$ so that $(P_k(w),Q_k(w))$ is nearly parallel to $(P_{[m,n)}(\bar z^2),\bar z P_{[m,n)}(-\bar z^2))$
so as to maximize the inner product. In this case we would have
\begin{align*}
 P_{[2^{k+1}m,2^{k+1}n)}(w)
 &=\big\langle (P_k(w),Q_k(w))^T, (P_{[m,n)}(\bar z^2),\bar zP_{[m,n)}(-\bar z^2))^T\big\rangle\\
 &\approx \|(P_k(w),Q_k(w))\|_2\cdot \|(P_{[m,n)}(\bar z^2),\bar zP_{[m,n)}(-\bar z^2))\|_2\\
 &=2^{(k+1)/2}\|P_{[m,n)}\|_L,
\end{align*}
and so $|P_{[2^{k+1}m,2^{k+1}n)}(w)|/2^{(k+1)/2}\approx \|P_{[m,n)}\|_L$ as required. Hence
it is enough to show that, for large~$k$, we can approximate any vector in the 3-sphere
\[
 S^3:=\{(\alpha,\beta)\in\C^2: |\alpha|^2+|\beta|^2=1\}
\]
with $(P_k(w)/2^{(k+1)/2},Q_k(w)/2^{(k+1)/2})$ for an appropriately chosen~$w$.

In \cite{R17} it was shown that for $w$ taken uniformly at random from $S^1:=\{w\in\C:|w|=1\}$ we have that
$P_k(w)/2^{(k+1)/2}$ converges in distribution as $k\to\infty$
to a uniform random variable in the unit disk $D:=\{z\in\C:z\le1\}$.
Indeed, a stronger theorem was proved. Let
\[
 G(w)=\frac{1}{\sqrt{2}}\begin{pmatrix}1&w\\1&-w\end{pmatrix}
\]
and note that
\[
 \begin{pmatrix}P_k(w)/2^{(k+1)/2}\\Q_k(w)/2^{(k+1)/2}\end{pmatrix}
 =G(w^{2^{k-1}})G(w^{2^{k-2}})\cdots G(w^2)G(w)\begin{pmatrix}2^{-1/2}\\2^{-1/2}\end{pmatrix}
\]
In \cite{R17} it is shown that if $w$ is distributed uniformly on $S^1$, then $G(w^{2^{k-1}})\cdots G(w)$ tends
in distribution to the Haar measure on the compact Lie group $U(2)$. This implies that
$(P_k(w)/2^{(k+1)/2},Q_k(w)/2^{(k+1)/2})$ tends in distribution to a uniform random variable in~$S^3$, and in particular
we can approximate any $(\alpha,\beta)\in S^3$ arbitrarily accurately as $k\to\infty$.

However, we also need the condition that $w^{2^k}=z$, so we cannot take $w$ to be uniform in~$S^1$. Fortunately
the proof in~\cite{R17} actually proves the following stronger statement.

\begin{theorem}\label{t:u2}
 For each\/~$k$, let\/ $w=w_k$ be drawn from a distribution\/ $\cD_k$ supported on\/ $S^1$ with the property that\/
 $\E(w^n)=0$ for all\/ $n$ such that\/ $2^k\ndv n$.
 Then\/ $G(w^{2^{k-1}})\cdots G(w)$ tends in distribution to the Haar measure on\/ $U(2)$ as\/ $k\to\infty$.
\end{theorem}

Indeed, the proof in~\cite{R17} follows by showing that, for any (finite dimensional) irreducible
representation $\pi$ of $U(2)$,
\begin{equation}\label{e:exp}
 \E\big[\pi(G(w^{2^{k-1}}))\dots \pi(G(w))\big]\to0\qquad\text{as }k\to\infty.
\end{equation}
Convergence in distribution to the Haar measure then follows from standard results (see Theorem~2.1 of~\cite{R17}).

For each fixed $\pi$, the expression inside the expectation in \eqref{e:exp} is a matrix
with entries that are polynomials in $w$ and~$w^{-1}$,
and the proof in~\cite{R17} proceeds by induction on~$k$, keeping only the terms $w^n$ with $2^k\mid n$ at each stage.
More specifically, assume $\pi$ is of dimension~$d$ and let $v\in\C^d$ be fixed. Then
\[
 \pi(G(w^{2^{k-1}}))\cdots \pi(G(w))v = \big(p^{(k)}_1(w),\dots,p^{(k)}_d(w)\big)^T
\]
where each $p^{(k)}_i(w)=\sum_{j=-2^kN}^{2^kM}p^{(k)}_{i,j}w^j\in \C[w,w^{-1}]$
for some fixed $N$ and $M$ depending only on~$\pi$. The coefficients $p^{(k+1)}_{i,2^kj}$
depend only on the coefficients $p^{(k)}_{\ell,2^km}$ as the entries of the $d\times d$ matrix $\pi(G(w^{2^k}))$
all lie in $\C[w^{2^k},w^{-2^k}]$. Thus by ignoring all terms $w^n$
with $2^k\ndv n$ in $p^{(k)}_i$ and dropping terms with $2^{k+1}\ndv n$ in $p^{(k+1)}_i$
we obtain a linear map
\[
 S\colon \C^{d(N+M+1)}\to\C^{d(N+M+1)};\qquad
 S\big(\big(p^{(k)}_{i,2^kj}\big)_{i,j}\big)=\big(p^{(k+1)}_{i,2^{k+1}j}\big)_{i,j}
\]
which is in fact easily seen to be \emph{independent\/} of~$k$. It is then shown that $\|S\|=\rho<1$ and so
$p^{(k)}_{i,2^kj}\to0$ as $k\to\infty$ for all $i,j$.

In \cite{R17} it is enough that the $j=0$ terms tend to 0 as for a uniform random variable on $S^1$
we have $\E(w^n)=0$ for all $n\ne0$. However, the proof clearly shows that the
whole vector $\big(p^{(k)}_{i,2^kj}\big)_{i,j}\in\C^{d(N+M+1)}$ tends to~0 as $k\to\infty$.
Thus if $\E(w^n)=0$ for all $n$ with $2^k\ndv n$ (and $|\E(w^n)|\le 1$ otherwise)
we see that for any $v\in\C^d$,
\[
 \big\|\E\big[\pi(G(w^{2^{k-1}}))\cdots \pi(G(w))v\big]\big\|_2\to0\qquad\text{as }k\to\infty.
\]
Hence \eqref{e:exp} holds and Theorem~\ref{t:u2} follows.

\begin{proof}[Proof of Theorem~\ref{t:dense}]
As shown above, it is enough to show that for any $(\alpha,\beta)\in S^3$, $\eps>0$, and $z\in S^1$ we can find,
for sufficiently large~$k$, a $w$ satisfying $w^{2^k}=z$ with
\[
 \big\|(P_k(w)/2^{(k+1)/2},Q_k(w)/2^{(k+1)/2})-(\alpha,\beta)\big\|_2<\eps
\]
For each $k$ we let $w=w_k$ be chosen uniformly at random from the solutions of $w^{2^k}=z$ and note
that $\E(w^n)=0$ for all $n$ with $2^k\ndv n$. Thus we can apply Theorem~\ref{t:u2}
to deduce that
\[
 \begin{pmatrix}P_k(w)/2^{(k+1)/2}\\Q_k(w)/2^{(k+1)/2}\end{pmatrix}
 =G(w^{2^{k-1}})\cdots G(w^2)G(w)\begin{pmatrix}2^{-1/2}\\2^{-1/2}\end{pmatrix}
\]
tends in distribution to the uniform measure on $S^3$ as $k\to\infty$. Thus for sufficiently
large~$k$, there is a positive probability that $(P_k(w)/2^{(k+1)/2},Q_k(w)/2^{(k+1)/2})$
lies in the ball of radius $\eps$ around $(\alpha,\beta)\in S^3$, and hence there exists
a solution $w$ of $w^{2^k}=z$ with this property.
\end{proof}

Finally we deduce~\eqref{e:max10} by estimating $\|P_{[m_k,n_k)}\|_L$. Recall that
\[
 m_k:=\frac{5\cdot 4^k+1}{3},\qquad n_k:=\frac{8\cdot 4^k+1}{3}.
\]
Thus $m_k=4m_{k-1}-1$ and $n_k=4n_{k-1}-1$. Also it is easy to check that $a_{m_k}=1$ and $a_{n_k}=-1$
for $k\ge 0$. Hence, by Proposition~\ref{p:1}(b),
\begin{align}
 P_{[m_{k+1},n_{k+1})}(z)
 &=P_2(z)P_{[m_k,n_k)}(z^4)+z^2Q_2(z)P_{[m_k,n_k)}(-z^4)+a_{m_{k+1}}z^{m_{k+1}}-a_{n_{k+1}}z^{n_{k+1}}\notag\\
 &=(1+z)P_{[m_k,n_k)}(z^4)+(z^2-z^3)P_{[m_k,n_k)}(-z^4)+z^{m_{k+1}}+z^{n_{k+1}}.\label{e:badrec}
\end{align}
Taking $z=\pm1$ we have $P_{[m_0,n_0)}(z)=P_{[2,3)}(z)=z^2=1$ and
\[
 P_{[m_{k+1},n_{k+1})}(1)=2P_{[m_k,n_k)}(1)+2,\quad\text{and}\quad P_{[m_{k+1},n_{k+1})}(-1)=2P_{[m_k,n_k)}(-1)-2.
\]
From this it follows by induction that
\begin{equation}\label{e:nmk1}
 P_{[m_k,n_k)}(1)=3\cdot 2^k-2\qquad\text{and}\qquad P_{[m_k,n_k)}(-1)=-2^k+2
\end{equation}
for all $k\ge0$. Thus
\begin{align*}
 \|P_{[m_k,n_k)}\|_L^2
 &\ge P_{[m_k,n_k)}(1)^2+P_{[m_k,n_k)}(-1)^2\\
 &\ge (3\cdot 2^k-2)^2+(-2^k+2)^2=10\cdot 2^{2k}-16\cdot 2^k+8.
\end{align*}
Thus in particular (as $n_k-m_k=4^k$)
\[
 \liminf_{k\to\infty}\frac{\|P_{[m_k,n_k)}\|_L^2}{n_k-m_k}\ge 10.
\]
As the ratio is always at most 10 by Theorem~\ref{t:twodim}, we deduce that \eqref{e:max10} holds.

To see an explicit case when $|P_{[m_k,n_k)}(z)|>3\sqrt{(n_k-m_k)}$ we can take
$z=e^{3\pi i/4}$. Then $z^4=-1$ and so \eqref{e:badrec} and \eqref{e:nmk1} imply
\begin{align*}
 P_{[m_{k+1},n_{k+1})}(z)
 &=(1+z)(-2^k+2) + (z^2-z^3)(3\cdot 2^k-2)+O(1)\\
 &=(-1-z+3z^2-3z^3)2^k+O(1).
\end{align*}
Hence
\[
 \frac{|P_{[m_k,n_k}(e^{3\pi i/4})|^2}{n_k-m_k}
 =\frac{1}{4}\big|-1-e^{3\pi i/4}+3e^{6\pi i/4}-3e^{9\pi i/4}\big|^2+o(1)=5+\tfrac{7}{\sqrt{2}}+o(1).
\]


\begin{thebibliography}{9}

\bibitem{A95} H. Alzer,
Note on an extremal property of the Rudin--Shapiro sequence.
\emph{Abh. Math. Sem. Univ. Hamburg} \bf 65 \rm (1995), 243--248.

\bibitem{B73} J. Brillhart,
On the Rudin--Shapiro polynomials,
\emph{Duke Math. J.}, {\bf 40} (1973), 335--353.

\bibitem{BC} J. Brillhart and L. Carlitz,
Note on the Shapiro polynomials,
\emph{Proc. Amer. Math. Soc.}, {\bf 25} (1970), 114--119.

\bibitem{BLM} J. Brillhart, J.S. Lomont and P. Morton,
Cyclotomic properties of the Rudin--Shapiro polynomials,
\emph{J. Rein. Angew. Math.}, {\bf 288} (1976), 37--65.

\bibitem{DH04} C. Doche and L. Habsieger,
Moments of the Rudin-Shapiro polynomials,
\emph{J. Fourier Anal. Appl.}, {\bf 10} (2004), 497--505.

\bibitem{G49} M.J.E. Golay,
Multislit spectrometry,
\emph{J. Opt. Soc. Am.} {\bf 39} (1949), 437--444.

\bibitem{MR} C. Mauduit and J. Rivat,
Prime numbers along Rudin--Shapiro sequences,
\emph{J. Eur. Math. Soc.}, {\bf 17} (2015), 2595--2642.


\bibitem{M17} H.L. Montgomery,
Littlewood polynomials,
In: \emph{Analytic Number Theory, Modular Forms and q-hypergeometric Series},
Springer Proc. Math. Stat., {\bf 221}, Springer, Cham, 2017, pp. 533--553.

\bibitem{R17} B. Rodgers,
On the distribution of Rudin--Shapiro polynomials and lacunary walks on $SU(2)$,
\emph{Adv. Math.}, {\bf 320} (2017), 993--1008

\bibitem{R59} W. Rudin,
Some theorems on Fourier coefficients,
\emph{Proc. Amer. Math. Soc.}, {\bf 10} (1959), 855--859.

\bibitem{S86} B. Saffari,
Une fonction extr\'emale li\'ee \`a la suite de Rudin--Shapiro,
\emph{C. R. Acad. Sci. Paris S\'er. I Math}, {\bf 303} (1986) 97--100.

\bibitem{S52} H.S. Shapiro,
Extremal problems for polynomials,
Thesis for S.M. Degree, MIT, 1952, 102 pp.
Avaliable at \url{https://dspace.mit.edu/handle/1721.1/12247}

\end{thebibliography}
\end{document}